\documentclass[12pt, leqno]{article}
\usepackage{amsmath,latexsym,amssymb,amsfonts}
\usepackage{enumerate}
\usepackage{amsthm}

\numberwithin{equation}{section} \topmargin 0cm

\newtheorem{theo}{Theorem}[section]
\newtheorem{lemma}[theo]{Lemma}
\newtheorem{prop}[theo]{Proposition}
\newtheorem{cor}[theo]{Corollary}
\newtheorem{defi}[theo]{Definition}
\theoremstyle{definition}
\newtheorem{rem}[theo]{Remark}
\newtheorem{exa}[theo]{Example}

\newcommand{\ds}{\displaystyle}
\newcommand{\Lp}{L^p ([-1,0],X)}
\newcommand{\xf}{\bigl(\begin{smallmatrix}
                           x\\
                           f
                       \end{smallmatrix}\bigr)}
\newcommand{\Wp}{W^{1,p}([-1,0],X)}
\newcommand{\sg}{$(T(t))_{t\geq 0}$ }
\newcommand{\sgs}{$(S(t))_{t\geq 0}$ }

\newcommand{\SG}{$(\cT(t))_{t\geq 0}$ }

\def\RR{{\mathbb{R}}}
\def\CC{{\mathbb{C}}}
\def\NN{{\mathbb{N}}}

\def\cL{{\mathcal{L}}}

\def\cT{{\mathcal{T}}}

\def\cA{{\mathcal{A}}}
\def\cE{{\mathcal{E}}}

\def\sg{\(\displaystyle{(T(t))_{t\geq 0}}\) }

\title
{\normalsize\bf \vskip 2truecm HYPERBOLICITY OF LINEAR PARTIAL
DIFFERENTIAL EQUATIONS WITH DELAY \footnote{TO APPEAR IN INTEGRAL
EQUATIONS AND OPERATOR THEORY} }
\author
{\normalsize ANDR\'AS B\'ATKAI \thanks{The author thanks  W. Desch
(Graz), I. Gy\H ori (Veszpr\'em) and R. Schnaubelt (Halle) for
helpful discussions.} }

\date{}
\begin{document}
\maketitle
\thispagestyle{empty}

\baselineskip=12pt
\begin{quote}
Robust hyperbolicity and stability results for linear partial
differential equations with delay will be given and, as an
application,  the effect of small delays to the asymptotic
properties of feedback systems will be analyzed.
\end{quote}

\vskip 1truecm
\baselineskip=15pt
\section{Introduction}

Partial differential equations with delay have been studied for
many years and by many different methods. In an abstract way and
using the standard notation (see  \cite{wu}), they can be written
as
\begin{equation*}
\begin{cases}
        u'(t)=Bu(t)+\Phi u_t ,&t\geq 0,\\
        u(0)=x, & \\
        u_0 =f,&
\end{cases} \tag{DE}
\end{equation*}
in a Banach space $X$, where $(B,D(B))$ is a (unbounded) linear
operator on $X$ and the delay operator $\Phi$ is supposed to
belong to, e.g., $\cL (W^{1,p}([-1,0],X) ,X)$ for some $1\leq
p\leq \infty$. J. Hale \cite{hale},  G. Webb \cite{webb}, N. Krasovski and
others were
 the first to apply semigroup theory to delay equations, and
we refer to  \cite{wu} for more recent references on partial
differential equations with delay.

As a first step one has to choose an appropriate state space. One
of the possibilities is to work in the space of continuous
$X$-valued functions. In this case, the relationship between
solutions of (DE) and a corresponding semigroup has been studied
intensively (see for example  \cite{ha-lu1}, \cite{wu} or
\cite[Section VI.6]{en-na}) and is well understood. On the other
hand, the state space $\cE :=X\times L^p([-1,0],X)$  turns out to
be a very good choice with regards to certain applications (e.g., to
control theory, see \cite{nak}, to numerical methods, see
\cite{kappel}), because we can use the reflexivity or the Hilbert space
structure of the state space. This approach will be used in this paper.

The aim of this work is to  give robust hyperbolicity and
stability results for linear partial differential equations with
delay, especially for the cases where no spectral mapping theorems
are available and we cannot use the powerful technics of
characteristic equations. As an application, we analyze the effect
of small delays to the asymptotic properties of feedback systems.

In the next section we collect some results on the semigroup
approach for delay equations in the $L^p$ history space, mainly
from \cite{ba-pi1}. This approach is especially useful in the
Hilbert space case because the theorem of  Gearhart is available and allows stability results in the case where
the semigroup generated by $(B,D(B))$ is not compact, see
\cite{ba-pi2} for applications.

In Section 3, we present robust hyperbolicity results in the
Hilbert space context  of the following kind. Assume, that
$(B,D(B))$ generates a hyperbolic semigroup and that the delay
operator $\Phi$ is "small" in some sense, which will be explaind in the text
later on. Then the delay semigroup
remains hyperbolic. As a special case we consider uniform
exponential stability.

In the last section we investigate the important question of the
effect of small delays. The problem is the following: We consider
delays of the special form $\Phi:=C\delta_{-\tau}$. The question
is, knowing that the solutions of the system are exponentially
stable for $\tau=0$, whether it follows that they remain stable
for arbitrary small $\tau>0$. This question is motivated by feedback-systems and
control theory and we give more references on this question in the
text. First two examples are given to show how the stability can
be destroyed and then  a general approach to treat this question
is provided. The problem  is considered for norm continuous
semigroups and for commuting compact perturbations.

\section{The semigroup approach to delay equations}

Let us summarize here some results from \cite{ba-pi1} on the
semigroup approach to linear partial differential equations with
delay.

Consider the equation
\begin{equation*}
\text{(DE)}\qquad
\begin{cases}
        u'(t)=Bu(t)+\Phi u_t ,&t\geq 0,\\
        u(0)=x, & \\
        u_0 =f,&
\end{cases}
\end{equation*}
where
\begin{itemize}
\item $x\in X$, $X$ is a Banach space,
\item $B:D(B)\subseteq X\longrightarrow X$ is a linear, closed, and
densely defined operator,
\item $f\in \Lp$, $p\geq 1$,
\item  $\Phi:W^{1,p}([-1,0],X)\longrightarrow X$ is a linear, bounded
operator,
\item $u:[-1,\infty )\longrightarrow X$ and $u_t
:[-1,0]\longrightarrow X$ is defined by $u_t (\sigma ):=u(t+\sigma
)$.
\end{itemize}
\begin{defi}
We say that a function $u:[-1,\infty )\longrightarrow X$ is a
(classical) solution of (DE) if
\begin{enumerate}[(i)]
\item $u\in C([-1,\infty ),X)\cap C^1 ([0,\infty ),X)$,
\item $u(t)\in D(B)$ and $u_t \in W^{1,p}([-1,0],X)$
for all $t\geq 0$, and
\item $u$ satisfies (DE) for all $t\geq 0$.
\end{enumerate}
\end{defi}
To be able to solve (DE) by semigroup methods, we introduce the
Banach space $$\cE :=X\times \Lp $$ with an arbitrary product
norm, usually the $p$-norm,  and the operator
\begin{equation}\label{delaymatrix}
\cA :=\begin{pmatrix}
             B&\Phi\\
             0&\frac{d}{d\sigma}
        \end{pmatrix}
\end{equation}
with domain
\begin{equation}\label{delaydomain}
D(\cA ):=\left\{\xf\in D(B)\times W^{1,p}([-1,0],X)\ :\
f(0)=x\right\}.
\end{equation}

Consider now the abstract Cauchy problem
\begin{equation*}
\text{(ACP)}\qquad
\begin{cases}
        v'(t)=\cA \,v(t) , &t\geq 0,\\
        v(0)=v_0 &
\end{cases}
\end{equation*}
associated to the operator matrix $(\cA ,D(\cA ))$ on the Banach
space $\cE$ with initial value $v_0:=\xf$. There is a natural
correspondence between the solutions of the two problems (see
\cite[Proposition 2.3 and 2.4]{ba-pi1}).
\begin{lemma}\label{cara}
\begin{enumerate}[(i)]
\item If $u$ is a solution of (DE), then $t\mapsto\bigl(\begin{smallmatrix}
                           u(t)\\
                           u_t
                       \end{smallmatrix}\bigr)$
 is a solution of the equation (ACP).
\item If $t\mapsto\bigl(\begin{smallmatrix}
                           u(t)\\
                           v(t)
                       \end{smallmatrix}\bigr)$  is
a solution of (ACP), then $v(t)=u_t$ for all $t\geq 0$ and $u$ is
a solution of (DE).
\end{enumerate}
\end{lemma}

We can then give the following definition for  well-posedness.
\begin{defi}
We say that (DE) is {\it well-posed} if
\begin{enumerate}[(i)]
\item for every $\xf\in D(\cA )$ there is a unique solution $u(x,f,\cdot )$,
and
\item the solutions depend continuously on the initial values, i.e.,
if a sequence $\bigl(\begin{smallmatrix}
                                               x_n\\
                                               f_n
                                          \end{smallmatrix}\bigr)$
in $D(\cA )$ converges to $\xf\in D(\cA )$,  then $u(x_n ,f_n ,t)$
converges to $u(x,f,t)$ uniformly for $t$ in compact intervals.
\end{enumerate}
\end{defi}

There is also a correspondence between the well-posedness of
equation (DE) and of the abstract Cauchy problem (ACP), see also
\cite[Theorem 2.8]{ba-pi1}.
\begin{prop}\label{caratterizzazione}
Let $(\cA ,D(\cA ))$ be the operator matrix defined by
(\ref{delaymatrix}) and (\ref{delaydomain}). Then the following
assertions are equivalent.
\begin{enumerate}[(i)]
\item Equation (DE) is well-posed.
\item $(\cA ,D(\cA ))$ is the generator of a strongly continuous semigroup on $\cE$.
\end{enumerate}
\end{prop}
As a consequence of Lemma \ref{cara} and Proposition
\ref{caratterizzazione}, we have that if $(\cA ,D(\cA ))$
generates a strongly continuous semigroup \SG \!\!, then the
solutions of equation (DE) are given by the first component of the
function $t\mapsto \cT (t)\xf$ for $\xf\in D(\cA)$.

By means of the perturbation theorem of Miyadera-Voigt (see
\cite{miy, voigt1, voigt2} and \cite[Corollary III.3.16]{en-na}) one can formulate the following
sufficient condition for the well-posedness of (DE), see
\cite[Theorem 3.3, Examples 3.4]{ba-pi1}, \cite{ma-vo}.

\begin{cor} \label{esempi}
Assume that $(B,D(B))$ generates a strongly continuous semigroup
\sgs on $X$, $\infty>p\geq 1$, and that there exists a function
$\eta :[-1,0]\rightarrow \cL (X)$
 of bounded variation such that  $\Phi :C([-1,0],X)\rightarrow X$ is given by the Riemann-Stieltjes integral
\begin{equation} \label{phi}
\Phi(f):=\int_{-1}^0 d\eta f.
\end{equation}
Then $(\cA ,D(\cA ))$ is a generator on $\cE$.
\end{cor}

An important special case is the operator $\Phi$ defined by
\begin{equation*}
\ds{\Phi (f):=\sum_{k=0}^{n}B_k f(h_k )},\quad f\in\Wp,
\end{equation*}
where $B_k \in\cL(X)$ and $h_k \in [-1,0]$ for $k=0,\ldots ,n$.

It was also shown in \cite{ba-pi1} that the class of delay
operators considered in Corollary \ref{esempi} satisfies the
following.

\begin{defi} \label{admissible}
We call the delay operator $\Phi\in \cL(W^{1,p}([-1,0],X),X)$
\textit{admissible} if
\begin{itemize}
\item[(a)] the operator $(\cA, D(\cA))$ is a generator for each generator $(B,D(B))$ and
\item[(b)] the function $\lambda\mapsto \Phi R(\lambda,A_0)$ is a bounded analytic function on the halfplane $\{\lambda\in \CC:\Re \lambda>\omega\}$ for all $\omega\in \RR$.
\end{itemize}
\end{defi}

We now characterize the resolvent set and the resolvent operator
of $\cA$ (see \cite[Lemma 4.1]{ba-pi1}). Here
$\epsilon_{\lambda}(t):=e^{\lambda t}$ and
$\Phi_{\lambda}\in\cL(X)$ is defined by
$\Phi_{\lambda}x:=\Phi(\epsilon_{\lambda}\times
Id)x=\Phi(e^{\lambda\cdot}x)$ for $x\in X$. The operator
$(A_0,D(A_0))$ is the generator of the nilpotent left shift
semigroup $(T_0(t))_{t\geq 0}$ in $\Lp$.
\begin{lemma}\label{spectrum}
Let $X$ be a Banach space, $(B,D(B))$ be linear, closed and
densely defined, and $\Phi:\Wp\longrightarrow X$ be linear and
bounded. Let $(\cA, D(\cA))$ be the operator matrix defined in
(\ref{delaymatrix}) and (\ref{delaydomain}). Then $\lambda\in\rho
(\cA )$ if and only if $\lambda\in\rho (B+\Phi_{\lambda })$.
Moreover, for $\lambda\in\rho (\cA )$ the resolvent $R(\lambda
,\cA )$ is given by
\begin{equation}\label{risolvente}
\begin{pmatrix}
R(\lambda ,B+\Phi_{\lambda }) & R(\lambda ,B+\Phi_{\lambda })\Phi
R(\lambda ,A_0 )\\ \epsilon_{\lambda }\otimes R(\lambda
,B+\Phi_{\lambda })& [\epsilon_{\lambda }\otimes R(\lambda
,B+\Phi_{\lambda})\Phi +Id]R(\lambda ,A_0 )
\end{pmatrix}.
\end{equation}
\end{lemma}

\section{Hyperbolicity and stability}

In the following, assume that $X$ is a Hilbert space and  $p=2$.
It follows that $\cE$ will be also a Hilbert space and we may use
the powerful Gearhart spectral mapping theorem (see e.g. \cite{gear, herb,
pruss}) to characterize
hyperbolicity and exponential stability of the delay semigroup.
The theorem can be found, e.g., in \cite[Theorem V.1.11]{en-na}
and \cite[Theorem V.1.18]{en-na} in the form we quote it. For the hyperbolicity, this
means that a semigroup \sg with generator $(G,D(G))$ in a Hilbert
space $X$ is hyperbolic if and only if $i\RR\subset \rho(G)$ and
$\sup\{\|R(i\omega,G)\|\,:\, \omega \in \RR\}<\infty$. There are
recent generalizations of this result to the Banach space case,
see \cite{ka-lu, la-sh}, which may  allow in the near future to
generalize  the results presented here to the Banach space case.

\begin{theo} \label{gen.hyp}
Let $X$ be a Hilbert space and consider the equation
\textnormal{(DE)}. Assume that $\Phi$ is admissible, the semigroup
$(B,D(B))$ generates a hyperbolic semigroup and consider
\begin{equation} \label{gen.cond.h.1}
a_{n}:=\sup_{\omega\in\RR}\left\|(\Phi_{i\omega}R(i\omega,B))^n\right\|<\infty.
\end{equation}
If
\begin{equation} \label{gen.cond.h.2}
a:=\sum_{n=0}^{\infty}a_{n}<\infty,
\end{equation}
then $(\cA,D(\cA))$ generates a hyperbolic semigroup.
\end{theo}

\begin{proof}
As a consequence of the above mentioned Theorem of Gearhart, the
numbers $a_n$ are defined for all $n\in \NN$ and we have to show
the boundedness of the resolvent operator given in
(\ref{risolvente}) on the line $i\RR$. Under our assumptions, this
is equivalent to the existence and boundedness of $R(\lambda
,B+\Phi_{\lambda})$ on the line $\{\lambda\in i\RR\}$.

Defining $M:=\sup_{\lambda\in i\RR}\|R(\lambda,B)\|$, we obtain
for all $\lambda\in i\RR$  that
\begin{equation*}
R(\lambda,B)\sum_{n=0}^{\infty}\left(\Phi_{\lambda}R(\lambda,B)\right)^n\in
\cL(X)
\end{equation*}
and
\begin{equation*}
\left\|R(\lambda,B)\sum_{n=0}^{\infty}\left(\Phi_{\lambda}R(\lambda,B)\right)^n\right\|\leq
M\sum_{n=0}^{\infty}\|\left(\Phi_{\lambda}R(\lambda,B)\right)^n\|\leq
 M\sum_{n=0}^{\infty}a_n = M\cdot a.
\end{equation*}
Easy calculations show that this operator defines an inverse for
$(\lambda-B-\Phi_{\lambda})$ being bounded on the line $i\RR$.
\end{proof}

\begin{cor}\label{cond.hyp}
Assume that $X$ is a Hilbert space, $\Phi$ is admissible,  and
$(B, D(B))$ generates a hyperbolic semigroup. If
\begin{equation}
\sup_{\omega\in\RR}\left\|\Phi_{i\omega}R(i\omega ,B) \right\|<1,
\end{equation}
or in particular if
\begin{equation}
\sup_{\omega\in\RR}\left\|\Phi_{i\omega}
\right\|<\frac1{\sup_{\omega\in\RR}\left\|R(i\omega ,B)\right\|},
\end{equation}
then $(\cA, D(\cA))$ generates a hyperbolic semigroup.
\end{cor}
The proof is an easy consequence of the previous theorem using
Weierstrass' criterion on the convergence of infinite series.

As an important special case of hyperbolicity, we may now consider
uniform exponential stability.

\begin{theo} \label{gen.stab}
Let $X$ be a Hilbert space and consider the equation
\textnormal{(DE)}. Assume that $\Phi$ is admissible, that
$\omega_0(B)<0$, and consider
\begin{equation} \label{gen.cond.1}
a_{n}:=\sup_{\omega\in\RR}\left\|(\Phi_{i\omega}R(i\omega,B))^n\right\|<\infty.
\end{equation}
If
\begin{equation} \label{gen.cond.2}
a:=\sum_{n=0}^{\infty}a_{n}<\infty,
\end{equation}
then $\omega_0(\cA)<0$.
\end{theo}

\begin{proof}
It follows from the inequality
\begin{equation*}
\sup_{\omega\in\RR}\left\|(\Phi_{\alpha+i\omega}R(\alpha+i\omega,B))^n\right\|\leq
\sup_{\omega\in\RR}\left\|(\Phi_{i\omega}R(i\omega,B))^n\right\|
\end{equation*}
for all $\alpha\geq 0$, which is a consequence of the generalized
maximum principle, and from Theorem \ref{gen.hyp} that the
semigroup generated by $(\cA-\alpha)$ is hyperbolic for all
$\alpha\geq 0$. Thus, $\omega_0(\cA)<0$.
\end{proof}

\begin{cor}\label{cond.stabilita}
Assume that $X$ is a Hilbert space, $\Phi$ is admissible,
$\omega_0 (B)<0$ and let $\alpha\in(\omega_0 (B),0]$. If
\begin{equation}
\sup_{\omega\in\RR}\left\|\Phi_{\alpha +i\omega}R(\alpha +i\omega
,B) \right\|<1,
\end{equation}
or in particular if
\begin{equation}
\sup_{\omega\in\RR}\left\|\Phi_{\alpha +i\omega}
\right\|<\frac1{\sup_{\omega\in\RR}\left\|R(\alpha +i\omega
,B)\right\|},
\end{equation}
then $\omega_0 (\cA )<\alpha\leq 0$.
\end{cor}

We demonstrate in the following example that the results obtained
in Theorem \ref{gen.stab} are more general then the ones in
\cite{ba-pi1}.
\begin{cor}
Assume that $X$ is a Hilbert space, $\Phi=C\delta_{-1}$ for
$C\in\cL(X)$ commuting with $(B,D(B))$, that $\omega_0(B)<0$ and
that
\begin{equation*}
r(C)<\frac1{\sup_{\omega\in\RR}\|R(i\omega,B)\|}.
\end{equation*}
Then $\omega_0(\cA)<0$.
\end{cor}

\begin{proof}
We use ideas analogous to \cite[Theorem IV.3.6]{kato}. Let us
denote by $M:=\sup_{\omega\in\RR}\|R(i\omega,B)\|$. Our assumption
means that there exists $0<q<1$ such that $r(C)\cdot M<q<1$. We
obtain that there exists $n_0\in\NN$ such that
\begin{equation*}
\|C^n\|^{\frac1n}\cdot M<q<1 \qquad \text{for all } n\geq n_0.
\end{equation*}
This means that
\begin{equation*}
a_{0,n}=\sup_{\omega\in\RR}\|(\Phi_{i\omega}R(i\omega,B))^n\| \leq
\sup_{\omega\in\RR}\|R(i\omega,B)^n\|\sup_{\omega\in\RR}\|\Phi_{i\omega}^n\|\leq
\|C^n\|\cdot M^n < q^n
\end{equation*}
for $n\geq n_0$, and the assertion follows by Weierstrass'
criterion.
\end{proof}

\section{An application: the effect of small delays}

The problem considered in this section is the following: Assume
that hyperbolicity or uniform exponential stability is known for
the solutions of the equation
\begin{equation*}
\text{(DE)}_{0}\qquad
\begin{cases}
        u'(t)=(B+C)u(t) ,&t\geq 0,\\
        u(0)=x, & \\
        u_0 =f,&
\end{cases}
\end{equation*}
where  $C\in \cL(X)$.

The question is, whether the same type of asymptotics holds for
the solutions of the equation
\begin{equation*}
\text{(DE)}_{\tau}\qquad
\begin{cases}
        u'(t)=Bu(t)+C u(t-\tau) ,&t\geq 0,\\
        u(0)=x, & \\
        u_0 =f,&
\end{cases}
\end{equation*}\index{$\text{(DE)}_{\tau}$}
where $\tau>0$ is ``small".

It is known, see e.g. \cite[Example B-IV.3.10]{nagel}, that if $X$
is a Banach lattice, $C$ is a positive operator and $(B,D(B))$
generates a positive semigroup, then the solutions of
(DE)$_{\tau}$ are uniformly exponentially stable if and only if
the solutions of (DE)$_0$ are uniformly exponentially stable.

It is an open question however, what happens to the hyperbolicity
in the positive case. The example of Montgomery-Smith
\cite{montgomery} suggests that this result may not remain true.

The first who examined this effect was R. Datko \cite{datko1,
datko2, datko3, DaLaPo, da-yo}. It is known for finite dimensional
equations that the stability cannot be destroyed and there exists
an extensive literature on delay dependent stability conditions,
see e.g. \cite{gyori1, gyori2}. For similar questions in the
parabolic case we refer to \cite{gu-ra-schn, schnaubelt}. There is
a recent exposition of this problem by J. Hale and S. Verduyn
Lunel \cite{ha-lu2, ha-lu3}, where many examples of functional
differential and difference equations are considered. A control
theoretical investigation using transfer functions was made for
compact feedback in \cite{ReTo}.

Before considering the abstract problem, we demonstrate on some
simple examples how the stability can be destroyed. Though the
following example seems to be known, we include it here because it
is the simplest example and we could not find it written  in the
literature.
\begin{exa}\label{tauk}
Let $(B,D(B))$ be the (unbounded) generator of a unitary group in
an infinite dimensional Hilbert space $H$ and let $C:=d\cdot Id$
for $d<0$. Then $(B+C,D(B))$ generates an exponentially stable
semigroup. We show that there exists a sequence $(\tau_k)$,
$\tau_k \rightarrow 0$, such that the solution semigroup of the
equation (DE)$_{\tau_k}$ does not decay exponentially for each
$k\in\NN$.

To construct this sequence, take $(\mu_k)\subset \RR$,
$i\mu_k\in\sigma(B)$ such that $|\mu_k|\rightarrow \infty$ and
$\mu_k\neq -d$. Defining the numbers
\begin{equation*}
\tau_k:=
\begin{cases}
\frac{3\pi}{2(\mu_k+d)}, & \mu_k+d >0, \\ \frac{-\pi}{2(\mu_k+d)},
& \mu_k+d<0,
\end{cases}
\end{equation*}
and the operators $\Phi^{(k)}:= d \cdot Id \delta_{-\tau_k}$,
$\Phi^{(k)}_{\lambda}:= e^{-\lambda \tau_k} d \cdot Id$, we obtain
for the numbers $\lambda_k:=(\mu_k+d)i\in i\RR$ that
\begin{equation*}
\lambda_k\in \sigma(B+\Phi^{(k)}i_{\lambda_k}) = \sigma(B)+d\cdot
e^{-\lambda_k \tau_k} = \sigma(B)+d\cdot i.
\end{equation*}
By the spectral characterization in Lemma \ref{spectrum} it
follows that the associated operator $(\cA,D(\cA))$ can not
generate a uniformly exponentially stable semigroup. If we assume
further that $i\mu_k\in P\sigma(B)$, which e.g., is satisfied if
$(B,D(B))$ has compact resolvent, then we also find classical
solutions of (DE)$_{\tau_k}$ which are not decaying exponentially.
\end{exa}

The essence of this example can be formulated as follows.

\begin{theo}
Let $X$ be a Hilbert space and assume that $(B,D(B))$ generates a
strongly continuous semigroup such that there exists $\rho\in \RR$
and $\mu_k\in \RR$, $|\mu_k|\to \infty$ such that $\rho+i\mu_k\in
\sigma(B)$, i.e., the spectrum is unbounded along an imaginary
line. Then there exists $C\in\cL(X)$ and $(\tau_k)\subset\RR^+$,
$\tau_k\to 0$ such that $\omega_0(B+C)<0$ but the solutions of
\textnormal{(DE)$_{\tau}$} are not uniformly exponentially stable.
\end{theo}

\begin{proof}
As in the previous example, take $C:=-\mu\cdot Id$, $\mu>0$,
$\mu>-\rho$ and define
\begin{equation*}
\tau_k:=\frac{\pi}{|\mu_k|}
\end{equation*}
for some $k\in\NN$. Then the corresponding characteristic
equations are again
\begin{equation*}
z-\lambda+\mu e^{-z\tau_k} = 0, \qquad \lambda\in\sigma(B).
\end{equation*}

Now take $\lambda=\rho+i\mu_k$ and put $z=\varepsilon + i\mu_k$
for some $\varepsilon>0$. Then we obtain
\begin{equation*}
i\mu_k +\varepsilon -\rho-i\mu_k +\mu e^{-i\mu_k\tau_k}\cdot
e^{-\varepsilon\tau_k} = 0,
\end{equation*}
and hence
\begin{equation*}
\varepsilon=\mu e^{-\varepsilon \tau_k} + \rho.
\end{equation*}
Since $\mu>-\rho$, there exists a positive real solution
$\varepsilon$.
\end{proof}
Unfortunately, if the stabilizing operator is not the identity,
the preceding technique can be applied only with enormous
difficulty even in cases where the spectral mapping theorem holds.
This is because we have in general no easy characterization of
$\sigma(B+\Phi_{\lambda})$, see \cite[Example IV. 3.8]{kato}.

Turning our attention now to the general problem of small delays,
we use an idea similar to  \cite[Section 5.4 (4.9)]{ha-lu1} and
transform the equation (DE)$_{\tau}$ into
\begin{equation*}
u'(t)=(B+C) u(t)+C\left(u(t-\tau)-u(t)\right).
\end{equation*}
We use the equality
\begin{equation} \label{ut_eq}
u(t_2)-u(t_1)=\left[S(t_2-t_1)-Id\right]u(t_1)+\int_{t_1}^{t_2}S(t_2-s)Cu(s-\tau)
ds
\end{equation}
for $t_2>t_1\geq 0$ following from (DE)$_{\tau}$. Substituting
$t_1=t-\tau$ and $t_2=t$, we obtain that
\begin{equation*}
u(t)-u(t-\tau)=\left[S(\tau)-Id\right]u(t-\tau)+\int_{-\tau}^{0}
S(-s)Cu(t+s-\tau) ds.
\end{equation*}

Thus, (DE)$_{\tau}$ can be written in the form
\begin{multline} \label{DEtau}
u'(t)=(B+C)u(t)\\
-C\left(\left[S(\tau)-Id\right]u(t-\tau)+\int_{-\tau}^{0}
S(-s)Cu(t+s-\tau) ds\right).
\end{multline}
Defining
\begin{equation} \label{phi_C}
\Phi f:=-C\left[S(\tau)-Id\right]\delta_{-\tau}f -
\int_{-\tau}^{0} CS(-s)C\delta_{s-\tau} f ds,
\end{equation}
where $\delta_{r} \in \cL\left(\Wp,X\right)$ is given by
$\delta_{r}(f):=f(r)$ for $r\in[-1,0]$, we see that the previous
stability results in Corollary \ref{cond.stabilita} are applicable
to our original problem and that (DE)$_{\tau}$ has the form
\begin{equation}\label{Detau2}
u'(t)=\left(B+C\right)u(t)+\Phi u_t.
\end{equation}

In order to be able to apply the stability results of Corollary
\ref{cond.stabilita}, or the hyperbolicity results of Corollary
\ref{cond.hyp}, we have to calculate
\begin{multline}\label{phibc}
\Phi_{\lambda}R(\lambda,B+C)x=-C\left[S(\tau)-Id\right]e^{-\lambda\tau}R(\lambda,B+C)x
\\ - \int_{-\tau}^{0} CS(-s)Ce^{-\lambda (s-\tau)} R(\lambda,B+C)x
ds.
\end{multline}

Defining
\begin{equation}\label{i1}
I^{\omega}_1(\tau):=C\left[S(\tau)-Id\right]e^{-i\omega\tau}R(i\omega,B+C)
\end{equation}
and
\begin{equation}\label{i2}
I^{\omega}_2(\tau)x:=\int_{-\tau}^{0} CS(-s)Ce^{-i\omega (s-\tau)}
R(i\omega,B+C)x ds,
\end{equation}
it would be sufficient to show that there exists $\kappa>0$ such
that $\sup_{\omega\in\RR}\|I^{\omega}_i(\tau)\|<\frac12$ for
$i=1,2$ and all $\tau\in (0,\kappa)$. Then, using Corollary
\ref{cond.stabilita} or Corollary \ref{cond.hyp} and
(\ref{phibc}), the assertion follows since
$\sup_{\omega\in\RR}\|\Phi_{i\omega}R(i\omega,B+C)\|\leq
\sup_{\omega\in\RR}\|I^{\omega}_1(\tau)\|+\sup_{\omega\in\RR}\|I^{\omega}_2(\tau)\|<1$.

The estimate on $I^{\omega}_2$ is
\begin{equation}\label{I2}
\|I^{\omega}_2(\tau)\|\leq \tau \|C\|^2K\|R(i\omega,B+C)\|,
\end{equation}
where $K:=\sup_{0\leq t\leq 1}\|S(t)\|$. Since
$\|R(i\omega,B+C)\|$ is uniformly bounded for all $\omega\in\RR$,
there exists $\kappa_2>0$ such that for all $\tau\in (0,\kappa_2)$
the estimate $\sup_{\omega\in\RR}\|I^{\omega}_2(\tau)\|<\frac12$
holds.

The estimate on $I^{\omega}_1$ is
\begin{multline}\label{I1}
\|I^{\omega}_1(\tau)\|\leq
\|C\|\cdot\|(S(\tau)-Id)R(i\omega,B+C)\| \\ \leq
\|C\|\cdot\|(S(\tau)-Id)R(\lambda,B)\|\cdot\|(\lambda-B)R(i\omega,B+C)\|,
\end{multline}
where $\lambda>\max\left\{\omega_0(B),0\right\}$ is fixed.

Since  $\|(\lambda-B)R(i\omega,B+C)\|$ is independent of $\tau$,
we only have to consider the term $(S(\tau)-Id)R(\lambda,B)$.

But then it follows from
\begin{multline*}
\|(S(\tau)-Id)R(\lambda,B)\|\leq
\|S(\tau)\|(1-e^{-\lambda\tau})\|R(\lambda,B)\| +
\|\int_0^{\tau}e^{-\lambda s}S(s) ds\| \\ \leq \tau
K\left(\|R(\lambda,B)\||\lambda|+1\right)
\end{multline*}
that for every $\omega\in\RR$
\begin{equation}\label{conver}
\lim_{\tau\to
0}\|C\left[S(\tau)-Id\right]e^{i\omega\tau}R(i\omega,B+C)\|=0.
\end{equation}

Example \ref{tauk} shows that in general this convergence cannot
be uniform. This is the point where we need some extra
assumptions.

As we could see in Example \ref{tauk}, the unboundedness of the
spectrum of $(B,D(B))$ along imaginary axes may cause trouble if
we allow any stabilizing operator $C$. In the following result on
the independence of stability of small delays the spectrum of the
generator also plays an important role.

\begin{theo} \label{stab.anal.boun}
Assume that $(B,D(B))$ generates an immediately norm continuous
semigroup  and that the  semigroup generated by $(B+C,D(B))$ is
exponentially stable or hyperbolic in the Banach space $X$. Then
there exists $\kappa>0$ such that the solution semigroup of
\textnormal{(DE)}$_{\tau}$ is exponentially stable or hyperbolic,
respectively,  for all $\tau\in (0,\kappa)$. Thus, the stability
and the hyperbolicity is not sensitive to small delays.
\end{theo}

\begin{proof}

We have to show that the convergence in (\ref{conver}) is uniform
in $\omega$.

To this end we use the immediate norm continuity of the semigroup
generated by $(B+C,D(B))$, see \cite[Theorem III.1.16(i)]{en-na}.
An important consequence  is that
$\lim_{|\omega|\to\infty}\|R(i\omega,B+C)\|=0$, see
\cite[Corollary II.4.19]{en-na}. Thus, there exists $L>0$ such
that
\begin{equation*}
\|R(i\omega,B+C)\|<\frac1{2\|C\|(K+1)}\quad \text{for }|\omega|>L,
\end{equation*}
where $K:=\sup_{0\leq t\leq 1}\|S(t)\|$.

For $\omega\in[-L,L]$, we recall that the function
\begin{equation*}
(\omega,\tau)\mapsto\|C\left[S(\tau)-Id\right]e^{i\omega\tau}R(i\omega,B+C)\|
\end{equation*}
is uniformly continuous on $[-L,L]\times[0,1]$. Thus, there exists
$\kappa_1>0$ such that for all $\tau\in(0,\kappa_1)$ and for all
$\omega\in [-L,L]$
\begin{equation*}
\|C\left[S(\tau)-Id\right]e^{i\omega\tau}R(i\omega,B+C)\|<\frac12.
\end{equation*}
Combining these estimates we obtain the desired statement.

The proof can be finished by choosing
$\kappa:=\min\left\{\kappa_1,\kappa_2\right\}$.
\end{proof}
We make the remark that the results of the previous theorem remain
true also if $X$ is a Banach space. This is because under the
conditions of the theorem, the delay semigroup will be eventually
norm continuous, see \cite[Proposition 5.3]{ba-pi1}, and for
eventually norm continuous semigroups the spectral mapping theorem
holds, see \cite[Theorem IV.3.10]{en-na}.

Analogous results were obtained by R. Schnaubelt \cite{schnaubelt}
 for non-autonomous equations in the parabolic case.

In the previous theorem we gave a condition on the generator
$(B,D(B))$ without any restriction on the stabilizing operator
$C$. In the following we also provide a condition involving $C$.

\begin{prop}
Let $(B, D(B))$ be a generator of a strongly continuous semigroup
\sgs in the Hilbert space $X$, $C\in\cL(X)$ be a compact operator
commuting with $B$ and that the  semigroup generated by
$(B+C,D(B))$ is exponentially stable or hyperbolic in the Hilbert
space $X$. Then  there exists $\kappa>0$ such that the solution
semigroup of \textnormal{(DE)}$_{\tau}$ is exponentially stable or
hyperbolic for all $\tau\in (0,\kappa)$. Thus, the stability and
the hyperbolicity is not sensitive to small delays.
\end{prop}

\begin{proof}
Again we only have to show that the convergence in (\ref{conver})
is uniform in $\omega$. Using that $C$ commutes with $B$ and hence
with the semigroup, we obtain
\begin{equation*}
C\left[S(\tau)-Id\right]e^{i\omega\tau}R(i\omega,B+C) =
\left[S(\tau)-Id\right]e^{i\omega\tau}R(i\omega,B+C)C.
\end{equation*}
By our assumptions, the set $CB(0,1)\subset X$ is precompact in
$X$. The proof can be finished by using the fact that on compact
sets the strong and the uniform topology coincide.
\end{proof}

\begin{rem}
To show that estimating  $I^{\omega}_1$ and $I^{\omega}_2$ is not
sharp, consider the well-known example $B=0$ and $C=d\cdot Id$ for
$d<0$ in the Banach space $\CC$. The direct calculation of the
formula (\ref{phibc}) shows that the solutions of (DE)$_{\tau}$
are exponentially stable if $|d|\tau<1$.

However, applying directly the spectral characterization of Lemma
\ref{spectrum} and using that
$\sigma(B+\Phi_{\lambda})=\left\{d\cdot e^{-\lambda\tau}\right\}$,
we obtain the well-known and best possible  estimate (see \cite[p.
135]{ha-lu1}) that the solutions decay exponentially if
\begin{equation}
|d|\tau < \frac{\pi}2.
\end{equation}
\end{rem}

\vskip 1truecm


\bigskip

ELTE TTK, Department of Applied Analysis, Pf. 120, H-1518 Budapest, Hungary \\
{\tt e-mail:} batka@cs.elte.hu

AMS Classification Numbers: 34K05, 34K20, 47D06

\end{document}